\documentclass[10pt]{amsart}

\usepackage{amsmath}
\usepackage{amssymb}
\usepackage{amsfonts}
\usepackage{geometry}
\usepackage{url}

\usepackage{amsmath}
\usepackage{amssymb}
\usepackage{amsfonts}
\usepackage{geometry}
\usepackage{url}
\usepackage{setspace}
\usepackage{amsthm}
\usepackage[all,cmtip]{xy}
\usepackage{amsmath,amscd}

\usepackage{amsthm}
\newtheorem{theorem}{Theorem}[section]
\newtheorem{lemma}[theorem]{Lemma}
\newtheorem{proposition}[theorem]{Proposition}
\newtheorem{corollary}[theorem]{Corollary}

\theoremstyle{definition}
\newtheorem{definition}[theorem]{Definition}

\newtheorem{conjecture}[theorem]{Conjecture}
\newtheorem{remark}[theorem]{Remark}

\newcommand{\gr}{\mathrm{gr}}
\newcommand{\Spec}{\mathrm{Spec}}
\newcommand{\Pid}{\mathrm{\langle P \rangle}}
\newcommand{\bd}{\mathrm{\bold{d}}}

\begin{document}

\title{On the Lower Central Series Quotients of a Graded Associative Algebra}

\author{Martina Balagovi\' c \and Anirudha Balasubramanian}
\address{Department of Mathematics,  Massachusetts Institute
of Technology, Cambridge, MA 02139, USA }
\email{martinab@math.mit.edu, anirudhab@gmail.com}
\maketitle

\maketitle

\begin{abstract} We continue the study of the lower central series $L_{i}(A)$ and its successive quotients $B_{i}(A)$ of a noncommutative associative algebra $A$, defined by $L_{1}(A)=A, L_{i+1}(A)=[A,L_{i}(A)]$, and $B_{i}(A)=L_{i}(A)/L_{i+1}(A)$. We describe $B_{2}(A)$ for $A$ a quotient of the free algebra on two or three generators by the two-sided ideal generated by a generic homogeneous element. We prove that it is  isomorphic to a certain quotient of K\"{a}hler differentials on the non-smooth variety associated to the abelianization of $A$. 

\end{abstract}

\section{Introduction}

For an associative algebra $A$, define its lower central series by $$L_{1}(A)=A, L_{i+1}(A)=[A,L_{i}(A)].$$ We are interested in the successive quotients of these subspaces $B_{i}(A)=L_{i}(A)/L_{i+1}(A),$ more precisely in $B_2(A)$.

The study of these quotients began in \cite{FS}, where the interest was focused on the case where $A=A_n$ is the free algebra over $\mathbb{C}$ with $n$ generators. One of the main results of this paper was that the space $B_{2}(A_{n})$ is isomorphic, as a graded vector space, to the space $\Omega^{\mathrm{even}>0}_{\mathrm{closed}}(\mathbb{C}^n)$ of closed differential forms of positive even degree on the space $\mathbb{C}^{n}$, the algebraic variety associated to the abelianization $\mathbb{C}[x_{1},\ldots x_{n}]$ of the algebra $A_{n}$. We will call this map the FS isomorphism. 

Next, \cite{DKM} provided an explicit basis for $B_2(A_n)$ and a new proof of the isomorphism with differential forms, along with partial results about $B_3(A_2)$. The appendix of the paper, by P. Etingof, studies $B_{2}(A)$ for any associative algebra $A$. There it is proved that if the variety associated to the abelianization $A_{ab}$ of $A$ is smooth, and certain other mild conditions are satisfied, then there exists an equivalent of the FS isomorphism between $B_{2}(A)$ and $\Omega^{\mathrm{even}>0}_{\mathrm{closed}}(\Spec(A_{ab}))$.


This paper studies the case when $A=A_{n}/\langle P \rangle$ is a quotient of the free algebra on $n$ generators by an ideal $\left< P \right>$ generated by a single homogeneous relation $P$. The abelianization of this algebra is no longer smooth, as the origin is not a smooth point if $\text{deg}(P)>1$. Hence, results from the appendix of \cite{DKM} no longer apply. We attempt to determine the extent to which singularities influence the structure of the $B_{2}$. 

Our main result is an analogue to the FS isomorphism:
\begin{theorem}
For generic $P$, $n=2,3$, the algebra $A=A_{n}/\langle P \rangle$, $\Omega^0=A_{ab}$, $\Omega ^1$ the module of K\"{a}hler differentials over $A_{ab}$, and $\mathbf{d}:\Omega^0\to \Omega^1$ the differential map, there exists an isomorphism of graded vector spaces
$$\phi: B_{2}(A) \to \Omega^{1}/{\bold d}\Omega^{0}.$$ \label{one}
\end{theorem}

Thus, the structure of $B_{2}(A_{n}/\left< P\right>)$ is not affected by the singularity for generic $P$ (although there are special polynomials $P$ for which the FS map is not an isomorphism, because the space $B_{2}(A_{n}/\left< P\right>)$ has graded pieces of larger dimension than the corresponding  $\Omega^{1}/{\bold d}\Omega^{0}$; examples include $n=2, 3,\,  P=x^2y$, or see Remark \ref{remark1}). This suggests that  a certain wider class of ``mild" singularities of $A_{ab}$ does not affect the structure of $B_{2}(A)$. While the proof that FS map is really an isomorphism from \cite{DKM} works only in the smooth case, and our proof works only for a special kind of singularity, it would be interesting to find a unified approach, and the one that would potentially extend this to a wider class of noncommutative algebras $A$.

The organization of the paper is as follows. Section 2 contains main definitions, some results from \cite{DKM}, and several technical propositions that we will use in calculations in the next sections. Sections 3 and 4 offer explicit bases of $B_{2}(A_{n}/\left< P \right>)$ for $P=x^d+y^d$ and $n=2, 3$ respectively. Section 5 uses the results of the appendix of \cite{DKM} to connect our case ($A=A_{n}/\left< P \right>$) and the corresponding smooth case ($A=A_{n}/\left< P-1 \right>$). Section 6 then uses the maps constructed in section 5 and data about dimensions in the special case of sections 3 and 4 to prove the main theorem for generic $P$. The Appendix contains some numerical data that we obtained with MAGMA algebra software \cite{Magma} in the early stages of the project. 

\subsection*{Acknowledgements} The authors are very grateful to Pavel Etingof for introducing us to this area of research, suggesting the problem, and devoting a great deal of energy to it through many helpful conversations. We are also grateful to David Jordan for help with the software MAGMA, and to Travis Schedler for explaining to us the results of \cite{EG}. In earlier versions here was an error in some data in Appendix A, which Teng Fei kindly pointed out; we thank him for that. The work of both authors was supported by the Research Science Institute, and conducted in the Department of Mathematics at MIT. The work of M.B. was partially supported by the NSF grant DMS-0504847.

\section{Preliminaries}

Throughout the paper, all the algebras are associative over $\mathbb{C}$ and with a unit. For $A$ such an algebra and $B,C$ subspaces of it, denote by $[B,C]$ the subspace of $A$ spanned by all the elements of the form $[b,c], \, b\in B, \, c\in C$.

Let us begin with definitions of the spaces upon which this paper will focus.
\begin{definition}
For any associative algebra $A$, define its {\it lower central series} to be the sequence $L_{i}(A)$ of subspaces of it defined inductively by $$L_{1}(A)=A \qquad L_{i+1}(A)=[A,L_{i}(A)],\, \,  i\ge 1.$$
Denote its successive quotients by $$B_{i}(A)=L_{i}(A)/L_{i+1}(A), i\ge 1.$$
\end{definition}

\begin{definition}
For $R$ a commutative associative algebra, let $\Omega^{0}(R)$ and  $\Omega^{1}(R)$  be $R$-modules as follows: $\Omega^{0}(R)=R$, and $\Omega^{1}(R)$ is defined by generators $\{{\bold d} f, f \in R\}$  and relations $${\bold d}(f\cdot g)=f\cdot {\bold d} g+g\cdot {\bold d} f  \qquad {\bold d}(\alpha f+\beta g)=\alpha{\bold d} f+\beta{\bold d}g\, \, \alpha, \beta\in \mathbb{C}.$$ 
The module $\Omega^{1}(R)$ is called the module of K\"{a}hler differentials of $R$. There is a natural map
 ${\bold d}:\Omega^{0}(R) \to \Omega^{1}(R)$ given by ${\bold d}(f)={\bold d} f$.  
\end{definition}

\begin{remark}If $\Spec R$ is smooth, $\Omega ^i(R)$ are exactly the spaces of regular differential $i$-forms on the algebraic variety $\Spec R$, with $\bold{d}$ the differential map. If additionally $\dim \Spec R\le 2$ and the space has zero deRham cohomology in positive degrees (e.g. $\mathbb{C}^2$), then the space of even closed forms is exactly $\Omega^1(R)/\bold{d}\Omega^0(R)$.  In this light, our results are an extension of the results of \cite{FS} and \cite{DKM}.
\end{remark}

For $A$ an associative algebra, denote by $A_{ab}$ its abelianization, i.e. the commutative associative algebra $A_{ab}=A/(A[A,A]A)$. Let  $A_{n}=\mathbb{C} \langle x_{1},...,x_{n}\rangle$ be the free algebra in $n$ generators, graded by $\deg x_{i}=1$. We will consider the quotient of $A_n$ by an associative ideal generated by one homogeneous relation $P$ of degree $d$. We will denote this ideal by $\langle P \rangle$ and the algebra by $A=A_{n}/\langle P \rangle.$ As $P$ is homogeneous, $A$ inherits the grading from $A_{n}$. In cases where we have few variables, we will call them  $x$, $y$, $z$, etc. instead of $x_1$, $x_2$, $x_3$, etc. 

We will denote the $i$-th component of any graded vector space $W$ by $W[i]$, and its Hilbert-Poincar\' e series by $h_{W}(t)=\sum_{i} \dim W[i]t^i$. 

For brevity, we shall denote $\Omega^{i}(\mathbb{C}[x_1,...x_n]/ \langle P \rangle)$ by $\Omega^{i}_{P=0}$. It will always be clear which $n$ we are considering. Note that the $\Omega^{i}$ are also graded by total degree of polynomial, and that ${\bold d}$ is a homogeneous map.

We shall now prove a few lemmata used extensively in the paper.

\begin{lemma}
\label{isos} The space $B_{2}(A_{n}/\langle P \rangle)$ is a quotient of $B_{2}(A_{n})$ by the image in $B_{2}(A_{n})$ of the intersection of $\langle P \rangle$ with $L_{2}(A_{n})$.

\begin{proof}
For brevity, write $L_{i}$ instead of $L_{i}(A_{n})$. Using the definition, the second and third isomorphism theorems, and the fact that the same holds for lower central series: $L_{i}(A_{n}/\Pid)=L_{i}(A_{n})/\left( L_{i}(A_{n})\cap \Pid \right)$, we get:
\begin{eqnarray*}
 B_{2}(A_{n}/\Pid) &=& \frac{L_{2}(A_{n}/ \Pid)}{L_{3}(A_{n}/\Pid)}\\ &=&  \frac{L_{2}/(L_{2} \cap \Pid)}{L_{3}/(L_{3} \cap \Pid)}\\ &=&  \frac{L_{2}/(L_{2} \cap \Pid)}{L_{3}/(L_{3} \cap \Pid \cap L_{2})}\\
 &=&\frac{L_{2}/(L_{2} \cap \Pid)}{(L_{3}+ \Pid \cap L_{2})/(L_{2}\cap \Pid  )}\\
 &=& L_{2}/(L_{3}+ \Pid\cap L_{2}),
\end{eqnarray*}
which is isomorphic to the quotient of $B_2(A_n)$ by the image of the intersection of the ideal $ \langle P \rangle$ with  $L_{2}(A_{n}),$ as claimed.
\end{proof}
\end{lemma}

This lemma allows us to use the explicit bases from Theorems 2.1 and 2.2 of \cite{DKM} in our context. Namely, we use it to conclude that the image under the quotient map of the basis for $B_{2}(A_{n})$ spans $B_{2}(A_{n}/\Pid)$, with the only new relations contained in $\Pid \cap L_{2}$.

Another result form \cite{DKM} generalizes easily to our context:

\begin{lemma}
\label{DKML}
For $q_{i}$, $Q$, $a$, $b$, and $c$ arbitrary elements of $A_n/\left< P \right>$, the following relations hold in $B_{2}(A_{n}/\langle P \rangle$:
\begin{enumerate}
\item $[Q, q_{1}q_{2}...q_{n}]=\sum_{i=1}^{n}\left[q_{i+1}q_{i+2}\ldots q_{n} Q q_{1}q_{2}\ldots q_{i-1} , q_{i}\right]$
\item $[ab,c]=[ba,c]$
\item $[a^{l}b,a^{k}]=\frac{k}{l+k}[b, a^{l+k}]$
\end{enumerate}
\end{lemma}

\begin{proof}
\cite{DKM}, proof of Proposition 2.1, proves these relations hold in $B_{2}(A_{n})$. Because of Lemma \ref{isos}, they hold in $B_{2}(A_{n}/\Pid)$.
\end{proof}

Most of the calculations in the paper will be concerned with algebras $A_n/ \langle P \rangle, n=2,3$ for the specific polynomial $P=x^d+y^d, d \ge 2$. Let us prove a useful proposition for that algebra:

\begin{proposition}
\label{proposition1}
Let $n,d\geq 2$, $j>0$, $i\geq0$. In $B_{2}(A_{n}/\langle x^{d}+y^{d}\rangle)$, for arbitrary $a \in A_{n}/\langle x^{d}+y^{d}\rangle$
 $$\mathrm{span} \{[x^{i+d}y^{j},a], [x^{i+d}a,y^{j}], [ay^{j},x^{i+d}]\}=\mathrm{span}\{[x^{i+d}y^{j},a]\}.$$
\end{proposition}

\begin{proof}
We will find two linearly independent equations relating $[x^{i+d}y^{j},a]$, $[x^{i+d}a,y^{j}]$, and $[ay^{j},x^{i+d}].$

Lemma \ref{DKML} (1) and (2) imply
\begin{equation}
\label{eq:Star1}
[x^{i+d}y^{j},a]+[x^{i+d}a,y^{j}]+[ay^{j},x^{i+d}]=0
\end{equation}

To obtain another such equation, use Lemma \ref{DKML} and $x^d+y^d=0$ to get
\begin{eqnarray*}
[x^{i+d}a,y^{j}] &=& -[y^{d}x^{i}a,y^{j}]\\
&=& -\frac{j}{j+d} [x^{i}a,y^{j+d}]\\
& = & \frac{j}{j+d}\left( [y^{j+d}x^{i},a]+[y^{j+d}a,x^{i}]\right) \\
& = & -\frac{j}{j+d}\left( [y^{j}x^{i+d},a]+[x^{d}y^{j}a,x^{i}]\right) \\
& = & -\frac{j}{j+d}[x^{i+d}y^{j},a]-\frac{j}{j+d}\cdot \frac{i}{i+d}[y^{j}a,x^{i+d}]
\end{eqnarray*}

This identity is not a multiple of identity (\ref{eq:Star1}) above, as $0\le \frac{i}{i+d}<1$ and $0<\frac{j}{j+d}<1$. Thus, we can eliminate two of the three terms of our initial spanning set. It is now easy to see that $[x^{i+d}a,y^{j}]$ and $[ay^{j},x^{i+d}]$ can always be expressed in terms of $[x^{i+d}y^{j},a]$.
\end{proof}


\section{Description of $B_{2}(A_{2}/\langle x^d+y^d\rangle)$}
We now construct a basis for $B_{2}(A_{2}/\langle x^d+y^d\rangle)$ and calculate the Hilbert-Poincar\' e series for the space.

\begin{theorem}\label{basis2}
The set $\{[x^i,y^j]: 0<i,j<d \}$ constitutes a basis for $B_{2}(A_{2}/\langle x^d+y^d\rangle)$. In particular, $\dim{B_{2}(A_{n}/\langle x^d+y^d\rangle)}=(d-1)^{2}.$
\end{theorem}

\begin{proof}

We shall first show that this set spans $B_{2}(A_{2}/\langle x^d+y^d\rangle)$ and then prove its linear independence.

By Proposition 2.1 of \cite{DKM}, we have that $\{[x^i,y^j]: i,j>0\}$ is a spanning set for $B_{2}(A_{2})$. Thus it will also span $B_2(A_{2}/\langle x^d+y^d \rangle)$, as $B_{2}(A_{2}/\langle x^d+y^d \rangle)$ is a quotient of $B_2(A_2)$ by Lemma \ref{isos}. 

\begin{lemma}
\label{nilp}
If $i \geq d$ or if $j \geq d$, then $[x^i,y^j]=0$.
\end{lemma}

\begin{proof}
We will only show the case $i \geq d$; the proof for $j \geq d$ is exactly the same.
Write $i=d+u,$ for $u\geq0$. 

Using $x^d+y^d=0$ and Lemma \ref{DKML}(3), we have
\begin{eqnarray*}
[x^i,y^j]&=&[x^{d}x^{u},y^j]\\
&=&-[y^{d}x^{u},y^{j}]\\
&=&-\frac{j}{j+d}[x^{u},y^{d+j}]\\
&=&-\frac{j}{j+d}[x^{d}y^{j},x^{u}]\\
&=&\frac{j}{j+d} \frac{u}{u+d}[x^{d}x^{u},y^{j}]\\
&=&\frac{j}{j+d} \frac{u}{u+d}[x^{i},y^{j}].
\end{eqnarray*}
As $1\ne\frac{j}{j+d}\cdot\frac{u}{u+d}$, it follows that $[x^{i},y^{j}]=0$.
\end{proof}

So, $\{[x^i,y^j]: 0<i,j<d \}$ is really a spanning set for $B_{2}(A_{2}/\Pid)$. Now we shall demonstrate that it is indeed a basis. Suppose that some linear combination of $[x^{i},y^{j}],\, \, 0<i,j<d$, is zero in $B_{2}(A_{2}/\langle x^{d}+y^{d} \rangle)$. By Lemma \ref{isos}, this linear combination, seen as an element of $L_{2}(A_2)$, belongs to $L_{3}(A_2)$+$A_2\cdot (x^{d}+y^{d}) \cdot A_2$. In other words, it can be expressed as a linear combination of triple brackets of monomials in $A_2$ plus some element $ f \in \langle x^d+y^d \rangle.$ We observe, however, that terms that come from the combination of $[x^{i},y^{j}]$ have degree $<d$ in both $x$ and $y$, while nonzero terms that come from the ideal always have degree either in $x$ or in $y$ (or in both) strictly larger than $d$. Finally, we observe that all four terms that come from each triple bracket of monomials have the same degrees in both $x$ and $y$. Hence, triple brackets of monomials which cancel the terms from the ideal cannot influence the left hand side, and vice versa. But then we can cancel out the ideal's contribution and assert that the linear combination of $[x^{i},y^{j}]$ lies in $L_{3}(A_{2})$. This, however, constitutes a contradiction, because $[x^{i},y^{j}]$ were linearly independent in $B_{2}(A_{2})$. 
\end{proof}

We will later prove that there is an isomorphism between $B_{2}(A_{n}/\Pid)$ and $\Omega^{1}_{x^d+y^d=0}/{\bold d}\Omega^{0}_{x^{d}+y^{d}=0}$. For  now, we can state:

\begin{corollary}
\label{isomo2}
The Hilbert-Poincar\' e series of  $B_{2}(A_{2}/\langle x^d+y^d \rangle)$ and $\Omega^{1}_{x^d+y^d=0}/{\bold d}\Omega^{0}_{x^{d}+y^{d}=0}$ coincide and are given by  $\left(\frac{t-t^{d}}{1-t}\right)^{2}.$
\end{corollary}

\begin{proof}
Using Theorem \ref{basis2},
$$h_{B_{2}(A_{2}/\langle x^{d}+y^{d}\rangle)}(t)=(t+t^2+\ldots t^{d-1})^2=\left( \frac{t-t^d}{1-t}\right)^2.$$

To calculate the Hilbert-Poincar\' e series for $\Omega^{1}_{x^{d}+y^{d}=0}/{\bold d}\Omega^{0}_{x^{d}+y^{d}=0}$, first note that $$\Omega^{0}_{x^{d}+y^{d}=0}=\mathbb{C}[x,y]/\langle x^{d}+y^{d}\rangle,$$ so $$h_{\Omega^{0}}(t)=\frac{1-t^{d}}{(1-t)^{2}}.$$ The only elements in the kernel of map $\bd$ are constants,  $$h_{{\bold d}\Omega^{0}}(t)=\frac{1-t^{d}}{(1-t)^{2}}-1.$$ Similarly, $$\Omega^{1}_{x^{d}+y^{d}}=(\mathbb{C}[x,y]{\bold d} x \oplus \mathbb{C}[x,y]{\bold d} y)/\langle x^{d}+y^{d}, x^{d-1}{\bold d} x+y^{d-1}{\bold d} y \rangle,$$ is a module over $\Omega^{0}$ with two homogeneous generators ${\bold d}x$ and ${\bold d}y$ of degree $1$ and one relation of degree $d$, namely  $x^{d-1}{\bold d}x+y^{d-1}{\bold d}y.$ Thus, $$h_{\Omega^{1}}(t)=\frac{1-t^{d}}{(1-t)^{2}}\cdot(2t-t^{d}).$$ Finally, $$h_{\Omega^{1}/{\bold d}\Omega^{0}}(t)=h_{\Omega^{1}}(t)-h_{{\bold d}\Omega^{0}}(t)=\left(\frac{t-t^{d}}{1-t}\right)^{2}.$$

Clearly, the series coincide.
\end{proof}

\section{Description of $B_{2}({A}_{3}/\langle x^{d}+y^{d}\rangle)$}

\begin{theorem}
\label{basisb2a3}
The following set constitutes a basis for $B_{2}(A_{3}/\langle x^{d}+y^{d}\rangle)$, with the constraints $0<i,j,k$ and $j<d$: 
$$ [x^{i},y^{j}], i<d, \qquad \qquad  [x^{i},z^{k}], \qquad \qquad  [y^{j},z^{k}], $$
$$[x^{i}y^{j},z^{k}],  \qquad \qquad [x^{i}z^{k},y^{j}],  i<d. $$
\end{theorem}

\begin{proof}
Theorem 2.2 of \cite{DKM} states the above list, subject only to the conditions $i,j,k>0$, constitutes a basis for $B_{2}(A_{3})$. By Lemma \ref{isos} it must be a spanning set for $B_{2}(A_{3}/\langle x^{d}+y^{d}\rangle)$.

Because $x^{d}+y^{d}=0$, we can express all elements as linear combinations of those that contain only powers of $y$ up to $d-1$, so we can add the constraint $j<d$. As in Lemma \ref{nilp} we can show that $[x^{i},y^{j}]=0$ if $i\geq d$. Finally, Proposition \ref{proposition1} with $a=z^k$ expresses any $[x^{i}z^{k},y^{j}]$ with $i\geq d$ in terms of $[x^{i}y^{j},z^{k}]$. Thus the above set is a spanning set for $B_{2}(A_{3}/\langle x^d+y^d\rangle).$

We now claim that this spanning set is indeed a basis. 

Assume that some linear combination of the elements above is 0 in $B_{2}(A_{3}/\langle x^{d}+y^{d}\rangle)$. Then, there exist $\alpha\in L_{3}(A_{3})$, a sum of triple brackets, and $\sum_{i}\beta_{i}(x^{d}+y^{d})\gamma_{i} \in A_3 \cdot(x^{d}+y^{d})\cdot A_3$ such that this linear combination equals $\alpha + \sum_{i}\beta_{i}(x^{d}+y^{d})\gamma_{i}$ in $A_{3}$. $\alpha$ is expressible as a linear combination of triple brackets of monomials; each such bracket expands into 4 monomials, all of which have the same degree in $x, y$, and $z$. Thus, any such triple bracket that affects the left-hand side has all $y$-degrees strictly less than $d$, and thus does not affect $\beta_{i}y^{d}\gamma_{i}$; vice-versa, we note that any triple bracket affecting $\beta_{i}y^{d}\gamma_{i}$ cannot affect the left hand side. Thus, all the $\beta_{i}y^{d}\gamma_{i}$ are canceled out by triple brackets. Alternatively stated, we $\sum_{i}\beta_{i}y^{d}\gamma_{i} \in L_{3}(A_{3})$. By symmetry, $\sum_{i}\beta_{i}x^{d}\gamma_{i} \in L_{3}(A_{3})$. Thus, our nontrivial linear combination of elements from the statement is in $L_{3}(A_{3})$. This would mean that these elements were linearly dependent in $B_{2}(A_{3})$, which is impossible, because they are a part of a basis of $B_{2}(A_{3}).$ 
\end{proof}

Again, we will eventually prove the existence of an isomorphism between the spaces for which we can now only say they have the same Hilbert-Poincar\' e series:
\begin{corollary}\label{isomo3}
The Hilbert-Poincar\' e series for $B_{2}(A_{3}/\langle x^{d}+y^{d}\rangle)$ and  $\Omega^{1}_{x^{d}+y^{d}=0}/{\bold d}\Omega^{0}_{x^{d}+y^{d}=0}$ coincide and are given by:
$$\frac{3t^{2}-t^{3}-3t^{d+1}+t^{2d}}{(1-t)^{3}}.$$ 
\end{corollary}

\begin{proof}
We can encode the size of the basis from Theorem \ref{basisb2a3} in a series:
$$h_{[x^{i}y^{j},z^{k}]}(t)=\sum_{l} \# \{[x^{i}y^{j},z^{k}] |\,  i+j+k=l, i, j, k >0,\, j>d\}\cdot t^{l}=\frac{t}{1-t}\cdot\frac{t-t^{d}}{1-t}\cdot\frac{t}{1-t}=\frac{t^{2}(t-t^{d})}{(1-t)^{3}}.$$
If we do so similarly for the others, we obtain the following expressions.\\
$$h_{[x^{i}z^{k},y^{j}]}(t)=\frac{t-t^{d}}{1-t}\cdot\frac{t}{1-t}\cdot\frac{t-t^{d}}{1-t}=\frac{t(t-t^{d})^{2}}{(1-t)^{3}}.$$ 
$$h_{[x^{i},y^{j}]}(t)=\frac{(t-t^d)^{2}}{(1-t)^{2}}.$$
$$h_{[y^{j},z^{k}]}(t)=\frac{(t-t^{d})t}{(1-t)^{2}}.$$
$$h_{[x^{i},z^{k}]}(t)=\frac{t^2}{(1-t)^{2}}.$$
Summing these equations, we conclude: 
$$h_{B_{2}(A_{3}/\langle x^{d}+y^{d}\rangle)}(t)=\frac{3t^{2}-t^{3}-3t^{d+1}+t^{2d}}{(1-t)^{3}}.$$

On the other hand, $\Omega^{0}_{x^{d}+y^{d}=0}=\mathbb{C}[x,y,z]/\langle x^{d}+y^{d}\rangle$ is a graded ring with three generators of degree one and a relation of degree $d$, so $$h_{\Omega^{0}}(t)=\frac{1-t^{d}}{(1-t)^{3}}.$$ The kernel of the differential map ${\bold d}$ again consists of just constants. Hence, $$h_{{\bold d}\Omega^{0}}(t)=\frac{1-t^{d}}{(1-t)^{3}}-1.$$ 

Next, $\Omega^{1}_{x^{d}+y^{d}=0}$ is a module over $\Omega^{0}_{x^{d}+y^{d}=0}$ with three generators ${\bold d}x, {\bold d}y$, and ${\bold d} z$ of degree one and one relation $x^{d-1}{\bold d}x+y^{d-1}{\bold d}y$ of degree $d$, so $$h_{\Omega^{1}}(t)=\frac{1-t^{d}}{(1-t)^{3}}\cdot(3t-t^{d}).$$ Finally, $$h_{\Omega_{1}/{\bold d}\Omega_{0}}(t)=h_{\Omega_{1}}(t)-h_{{\bold d}\Omega^{0}}(t)=\frac{3t^{2}-t^{3}-3t^{d+1}+t^{2d}}{(1-t)^{3}}.$$ So, the series coincide.
\end{proof}

\section{A Connection to the Smooth Case}

We want to prove that for $n=2,3$ and generic homogeneous $P$ there exists an analogue of the FS isomorphism, i.e. a map $$ B_{2}(A_{n}/\Pid) \to \Omega^{1}_{P=0}/{\bold d}\Omega^{0}_{P=0}.$$
If we replace $P$ by $P-1$, the resulting algebra has smooth abelianization. If $P$ is generic, then the algebra $A_{n}/\left< P-1 \right>$ satisfies the conditions of the appendix of \cite{DKM}, and there is an isomorphism $$\phi: B_{2}(A_{n}/\left< P-1 \right>) \to \Omega^{1}_{P=1}/{\bold d}\Omega^{0}_{P=1}.$$ We want to show that the associated graded map to this isomorphism is the isomorphism we want. To do this, we need to establish, for generic P, a relationship between the structures we get in the smooth case P-1=0 and in the graded case P=0.


\begin{lemma}
\label{Thm51}
For $P$ a generic homogeneous polynomial, there exists a surjection:
$$B_{2}(A_{n}/ \langle P \rangle) \twoheadrightarrow \gr B_{2}(A_{n}/ \langle P-1 \rangle).$$
\end{lemma}

\begin{proof}  
{\bf STEP 1: $A_n/ \langle P \rangle \cong \gr (A_n/ \langle P-1 \rangle)$. }

There is an obvious surjection $A_n/ \langle P \rangle \to \gr(A_n/ \langle P-1 \rangle)$. For generic $P$ the algebra $A_n/\langle P\rangle$ is an noncommutative complete intersection (NCCI) algebra as defined in \cite{EG}, Theorem 3.1.1. This can be seen for example from \cite{EG}, Theorem 3.2.4, as for a generic $P$ condition 2) from that theorem is satisfied. So, by Theorem 3.2.4(5) of \cite {EG}, its noncommutative Koszul complex is acyclic in higher degrees with the homology at $0$ being $A_{n}/\Pid$. If we build the analogous complex for the filtered algebra $A_n/\langle P-1\rangle$, its associated graded complex will be the noncommutative Koszul complex of $A_n/\langle P\rangle$. So, the complex of the filtered algebra is also acyclic in higher degrees, with zero degree homology $A_n/\langle P-1\rangle$. Therefore, 
$\gr(A_n/ \langle P-1 \rangle)\cong A_n/ \langle P \rangle$.



{\bf STEP 2: $L_{2}(A_{n}/\langle P \rangle \cong \gr L_{2}(A_{n}/\langle P-1\rangle).$}

 For $B$ a an algebra with an ascending filtration by nonnegative integers ($B_0\subset B_1\subset\ldots$), and $X,Y$ subspaces of $B$, there exists an injection $[\gr X ,\gr Y ] \hookrightarrow \gr [X,Y]$. If $P$ is generic, then by STEP 1 $\gr(A_n/ \langle P-1 \rangle)\cong A_n/ \langle P \rangle$. So, for $X=Y=A_n/ \langle P-1 \rangle$, we have:
 $$L_{2}( A_n/\langle P\rangle)=[A_n/\langle P\rangle, A_n/\langle P\rangle]=[\gr(A_n/\langle P-1 \rangle),\gr(A_n/\langle P-1 \rangle)] \hookrightarrow$$ $$\hookrightarrow \gr [A_n/\langle P-1 \rangle, A_n/\langle P-1 \rangle]=\gr L_{2}(A_n/\langle P-1 \rangle).$$
 
For generic $P$ the algebra $A_{n}/\langle P \rangle$ is an asymptotic representation complete intersection (asymptotic RCI), as in \cite{EG}, Definition 2.4.8. Then one can conclude from Theorem 3.7.7 in \cite{EG} that the Hilbert-Poincar\' e series of $B_{1}(A_{n}/\langle P \rangle)$ and $\gr B_{1}(A_{n}/\langle P-1 \rangle)$ coincide. This, together with STEP 1 and the existence of an injection from the beginning of STEP 2, implies that the injection is in fact an isomorphism.

{\bf STEP 3: $L_{3}(A_{n}/ \langle P \rangle) \hookrightarrow \gr L_{3}(A_{n}/ \langle P-1 \rangle).$}

For $X=L_{2}(A_{n}/\left< P-1 \right>)$, $Y=L_{1}(A_{n}/\left< P-1 \right>)$ we have:
$$L_{3}(A_{n}/ \langle P \rangle)=[L_{2}(A_{n}/ \langle P \rangle),L_{1}(A_{n}/ \langle P \rangle)]=[\gr L_{2}(A_{n}/ \langle P-1 \rangle), \gr L_{1}(A_{n}/ \langle P-1 \rangle)] \hookrightarrow $$
$$\hookrightarrow \gr[ L_{2}(A_{n}/ \langle P-1 \rangle), L_{1}(A_{n}/ \langle P-1 \rangle)]=\gr  L_{3}(A_{n}/ \langle P-1 \rangle).$$

Steps 2 and 3 now imply the statement.
\end{proof}

\begin{remark}\begin{enumerate}
\item It is useful to note that the isomorphism $\gr(A_{n}/\langle P-1\rangle)\cong A_{n}/\langle P \rangle$ from step 1 of the proof of Lemma  \ref{Thm51} does not hold for every $P$. For example, consider $A_{2}/\langle xyx\rangle.$ Then, $A_{2}/\langle xyx-1\rangle$ is commutative because $xy=xyxyx=yx$, so $A_{2}/\left< xyx-1\right>=\mathbb{C}[x,y]/\left< x^2y-1\right>$ is an algebra of polynomial functions on the curve $x^2y=1$, and has linear growth in its graded components. The algebra $A_{2}/\langle xyx\rangle$, on the other hand, has exponential growth in its graded components, as for any sequence of integers $m_1,m_2,\ldots,m_k$, with all $m_{i}\geq 2$, the elements $xy^{m_1} xy^{m_2}\ldots xy^{m_k}$ are linearly independent. Hence, the map from STEP 1 is surjective, but not injective.\label{xyx}
\item A sufficient condition for statement of  Lemma \ref{Thm51} is for  $P$  to be such that $A_{n}/\Pid$ an asymptotic RCI. An inspection of the proof shows that the only things we used about $A_{n}/\Pid$ is that it is an asymptotic RCI and NCCI, which follows from being an asymptotic RCI. For a more detailed discussion, see \cite{EG}.
\end{enumerate}
\label{remark1}
\end{remark}




\begin{theorem}
\label{Thm52}For generic homogeneous $P$ and $n=2,3$
$$B_{2}(A_{n}/ \langle P-1 \rangle)\cong \Omega^{1}_{P=1}/{\bold d}\Omega^{0}_{P=1}.$$
\end{theorem}

\begin{proof}
The requirement that $P$ is generic can be made more precise by requiring that $P$ satisfies:
\begin{enumerate} 
\item $A_{n}/ \langle P \rangle$ is an asymptotic RCI, 
\item $\mathbb{C}[x_1,..,x_n]/\langle P-1 \rangle$ is  smooth, 
\end{enumerate}

Proposition 7.8 and Theorem 7.2 of \cite{DKM} in our setting state that if condition (ii) is satisfied, then $$\Omega_{P=1}^{\mathrm{\mathrm{odd}}}/{\bold d}\Omega_{P=1}^{\mathrm{\mathrm{even}}} \twoheadrightarrow B_{2}(A_{n}/ \langle P-1 \rangle)$$ and the kernel of the map is zero if a certain homology  (see \cite{EG}), $HC_{1}((A_{n}/ \langle P-1 \rangle)_{ab})$ is zero.  
Condition (i) and Theorems 3.7.1 and 3.7.7 in  \cite{EG} guarantee that $HC_{1}((A_{n}/ \langle P \rangle_{ab})=0$. The complex that calculates $HC_{\bullet}((A_{n}/ \langle P-1 \rangle)_{ab})$ is filtered and its associated graded complex is the one calculating $HC_{\bullet}((A_{n}/ \langle P \rangle_{ab})$; so $HC_{1}((A_{n}/ \langle P \rangle_{ab})=0$ implies $HC_{1}((A_{n}/ \langle P-1 \rangle_{ab})=0$. 

Finally, because $n=2, 3$, 
 $$\Omega_{P=1}^{\mathrm{\mathrm{odd}}}/{\bold d}\Omega_{P=1}^{\mathrm{\mathrm{even}}}=\Omega^{1}_{P=1}/{\bold d}\Omega^{0}_{P=1}.$$
\end{proof}

\begin{lemma} 
\label{Thm53}
If $P$ has no repeated factors, $$\gr \left(\Omega^{1}_{P=1}/{\bold d}\Omega^{0}_{P=1} \right) \cong \Omega^{1}_{P=0}/{\bold d}\Omega^{0}_{P=0}.$$
\end{lemma}

\begin{proof}
We will show that $\gr \Omega^{0}_{P=1} =\Omega^{0}_{P=0}$, $\gr \Omega^{1}_{P=1} =\Omega^{1}_{P=0}$ and that $\gr \bold d=\bold d$. This implies the result.


For any graded commutative algebra $A$, any filtered $A$-module $M$ and any submodule $I$ in $M$, we can consider the associated graded modules $\gr M, \gr I, \gr (M/I)$  Dente the $m$-th filtered piece of $M$ by $M_{m}$, and the $m$-th graded piece of the associated graded by $\gr M[m]=M_{m}/M_{m-1}$. There is an isomorphism $$\gr (M/I)\cong \gr M / \gr I.$$

Let us apply this first to $A=\mathbb{C}[x_{1},\ldots x_{n}]$, $M=A$, $I=\left< P-1 \right>$. Then $$\gr \left< P-1 \right>=\gr A(P-1)\cong AP=\Pid.$$ The isomorphism $\psi: AP\to \gr A(P-1)$ is  sending $aP, a \in A[m]$ to $aP$ in the $m$-th graded piece of $\gr A(P-1)$. It is well defined as $aP$ maps to the image of the element $a(P-1)$ in the graded module $\gr A(P-1)$. For the same reason it is surjective. If $aP, a \in A[m]$, maps to $0$ in $\gr A(P-1)$, then $aP=0\in A(P-1)_{m}/A(P-1)_{m-1}$, so there exists $b\in A$ such that $b(P-1)\in A(P-1)_{m-1}$ and $aP+bP-b=0\in A(P-1)_{m}$. However, this means (looking at degrees) that $aP=0$. So, $$\gr \Omega^0_{P-1}=\gr M/I \cong M/\gr I =\mathbb{C}[x_{1},\ldots x_{n}]/\Pid =\Omega^0_{P}.$$

Next, do the same for $A=\mathbb{C}[x_{1},\ldots x_{n}]$, $M=\oplus_{i} A\bd x_{i}$, $I=\left< \bd P, P-1 \right>$. We claim that if $P$ has no double factors, then $\gr \left< \bd P, P-1 \right>= \left< \bd P, P\right>$. There exists a map as above $$\psi '':\left< \bd P, P\right>\to \gr \left< \bd P, P-1 \right>.$$ If $P$ has degree $d$, so does $\bd P$, and for any $a\in A[m-d]$, $b\in M[m-d]$, the element $a\bd P +Pb$ cannot be $0$ in $\gr \left< \bd P, P-1 \right>$ unless it is $0$ in $\gr \left< \bd P, P \right>$. So, the map is injective. To prove surjectivity, let $a\in A$, $b\in M$,  and consider $a\bd P+b(P-1)$. We want to show that the top degree part of this element is indeed in $\left< \bd P, P\right>$.

If $\deg a\ne \deg b$, this is true. Now assume $\deg a=\deg b=m$, and denote the $m$-th degree parts of them by $a_{m},b_{m}$. We will proceed by induction on $m$. The image of $a\bd P+(P-1)b$ in the graded module is either $a_{m}\bd P+Pb_{m}$, which is in $\left< \bd P, P\right>$, or $a_{m}\bd P+Pb_{m}=0$ and the image is in a lower degree then $m+d$. If $a_{m}\bd P+Pb_{m}=0$, then, using that $P$ and $\bd P$ have no common factors, there exists $c\in A$ such that $a_{m}=Pf, b_{m}=-f\bd P$. So we can write $a=Pf+\bar{a}, b=-f\bd P+\bar{b}$, with $\deg \bar{a}, deg \bar{b}<m$, and then $$a\bd P+b(P-1)=Pf\bd P+\bar{a}\bd P-fP\bd P +f\bd P +\bar{b}(P-1)=(\bar{a}+f)\bd P+\bar{b}(P-1).$$ As both $\bar{a}+f$ and $\bar{b}$ have lower degrees than $m$, we can conclude that the top degree of $(\bar{a}+f)\bd P+\bar{b}(P-1)$ is in $\left< \bd P, P\right>$ by previous argument or by induction assumption. So, 
 $$\gr \Omega^1_{P-1}=\gr \oplus_{i}\mathbb{C}[x_{1},\ldots x_{n}]\bd x_{i}/\left< \bd P, P-1\right> \cong\oplus_{i}\mathbb{C}[x_{1},\ldots x_{n}]\bd x_{i}/\left< \bd P, P\right>= \Omega^1_{P} .$$

The maps $\bd: \Omega^0_{P}\to \Omega^1_{P}$ and $\gr \bd: \gr \Omega^0_{P-1}\to \gr \Omega^1_{P-1}$ coincide on the generators, so they are the same. 
\end{proof}


\section{Main Result}
The following lemma will allow us to connect our Hilbert-Poincar\' e series computations to the results from Section 5. 


\begin{lemma} \label{goodP}
$P=x^{d}+y^{d}$ is generic enough in $A_{2}$ and in $A_{3}$ to satisfy claims of Lemma \ref{Thm51}, Theorem \ref{Thm52} and Lemma \ref{Thm53}; in other words for $n=2,3$, the algebra $A_{n}/ \langle P\rangle$ is an asymptotic RCI, $(A_{n}/ \langle P-1 \rangle)_{ab}$ is smooth, and $P\in (A_{n})_{ab}$ does not have repeated factors.
\end{lemma}

\begin{proof}
First we demonstrate that $A_{n}/ \langle P \rangle$ satisfies the sufficient condition for being an asymptotic RCI given by Theorem 5.4.1 in \cite{EG}. A change of variables $a=y-\xi x$, where $\xi^{d}=-1$, makes $x^{d}+y^{d}$ into a homogeneous polynomial of degree $d$ in $x$ and $a$. We can impose an ordering  on monomials in $x$ and $a$ by $M>M'$ if $\deg_{x}M>\deg_{x}M'$ or $\deg_{x}M=\deg_{x}M'$ and the sum of positions of $x$, counting from the left, is smaller for $M$ than for $M'$. (For example, $x^{2}a^{2}>xa^{3}$, and $x^{2}a^{2}>xaxa$.) In this ordering, the leading monomial of $x^{d}+(a+\xi x)^{d}$ is not $x^{d}$, which appears with coefficient $1+\xi^{d}$=0, but $x^{d-1}a$, which appears with the nonzero coefficient $\xi^{d-1}$. So the leading monomial of $P$ satisfies the conditions of Theorem 5.4.1 from \cite{EG} (it is ``non overlapping", meaning there does not exist a nontrivial sub-word $w$ that the monomial both begins and ends with). That implies that the quotient of $A_n$ by this leading monomial is an asymptotic RCI. But the quotient of $A_n$ by a leading monomial can be considered an associated graded algebra of the quotient of $A_n$ by the entire polynomial, where filtration of $A_n$ is given by the above ordering on monomials. However, if an associated graded algebra is an asymptotic RCI, then so is the original algebra $A_{n}/ \langle P \rangle$.

Because the system $$P=x^d+y^d=1 \qquad \frac{\partial}{\partial x}P=dx^{d-1}=0 \qquad \frac{\partial}{\partial y}P=dy^{d-1}=0$$ has no solution, $(A_{n}/ \langle P-1 \rangle)_{ab}$ is smooth. 

Finally, $P$, seen as a commutative polynomial, does not have repeated factors. Indeed, $P=x^d+y^d=(x-\xi y)(x-\xi^3 y)\ldots (x-\xi^{2d-1} y)$.
 
\end{proof}

\begin{theorem}
\label{OnePoly}
For $n=2,3$, and $P=x^{d}+y^{d}$ the associated graded map to the FS isomorphism
$$\phi: B_{2}(A_{n}/ \langle P-1 \rangle)\to \Omega^{1}_{P=1}/{\bold d}\Omega^{0}_{P=1}$$ is the isomorphism $$\gr \phi :  B_{2}(A_{n}/ \langle P \rangle)\to \Omega^{1}_{P=0}/{\bold d}\Omega^{0}_{P=0}.$$
\end{theorem}

\begin{proof}

Lemma \ref{Thm51}, Theorem \ref{Thm52}, Lemma \ref{Thm53} and Lemma \ref{goodP} give us a series of graded morphisms:

$$B_{2}(A_{n}/ \langle P \rangle) \twoheadrightarrow \gr(B_{2}(A_{n}/ \langle P-1 \rangle)) \cong  \gr(\Omega^{1}_{P=1}/{\bold d}\Omega^{0}_{P=1})  \cong \Omega^{1}_{P=0}/{\bold d}\Omega^{0}_{P=0}.$$

The first and the last map are natural, and the middle map is $\gr \phi$. Because the Hilbert-Poincar\' e series of $B_{2}(A_{n}/ \langle P \rangle)$ and $\Omega^{1}_{P=0}/{\bold d}\Omega^{0}_{P=0}$ are the same by Theorems \ref{isomo2} and \ref{isomo3}, the first map is also an isomorphism. Composing all maps, we get the isomorphism $\Omega^{1}_{P=0}/{\bold d}\Omega^{0}_{P=0}  \cong  B_{2}(A_{n}/ \langle P \rangle)$, which can, with proper identifications, be thought of as $\gr \phi$.
\end{proof}



\begin{theorem}\label{main}
For $n=2,3$, and generic homogeneous $P$ the associated graded map to the FS isomorphism $$\phi: B_{2}(A_{n}/ \langle P-1 \rangle)\to \Omega^{1}_{P=1}/{\bold d}\Omega^{0}_{P=1}$$ is the isomorphism $$\gr \phi :  B_{2}(A_{n}/ \langle P \rangle)\to \Omega^{1}_{P=0}/{\bold d}\Omega^{0}_{P=0}.$$
\end{theorem}

\begin{proof}
As in Theorem \ref{OnePoly}, for a generic $P$ and using Lemma \ref{Thm51}, Theorem \ref{Thm52}, Lemma \ref{Thm53}, we have the following series of graded morphisms:
$$B_{2}(A_{n}/ \langle P \rangle) \twoheadrightarrow \gr B_{2}(A_{n}/ \langle P-1 \rangle) \cong  \gr(\Omega^{1}_{P=1}/{\bold d}\Omega^{0}_{P=1})  \cong \Omega^{1}_{P=0}/{\bold d}\Omega^{0}_{P=0}.$$ This implies that 
\begin{equation}\dim B_{2}(A_{n}/ \langle P \rangle) [l]\ge \dim \Omega^{1}_{P=0}/{\bold d}\Omega^{0}_{P=0}[l] \end{equation} for all $l$. 
If $\deg P=d$, then \begin{equation}\dim  \Omega^{1}_{P=0}/{\bold d}\Omega^{0}_{P=0}[l]=\dim  \Omega^{1}_{x^d+y^d=0}/{\bold d}\Omega^{0}_{x^d+y^d=0}[l],\end{equation} as this is the same for all $P$ of the same degree. On the other hand, $\dim B_2(A_n/\langle P\rangle)[l]$ is going to be the same for generic $P$, and higher for special $P$. In this aspect, we can only say that for generic $P$, \begin{equation}\dim B_2(A_n/\langle x^d+y^d\rangle )[l] \ge \dim B_2(A_n/\langle P\rangle )[l].\end{equation} However, we have shown in Theorems \ref{isomo2} and \ref{isomo3} that \begin{equation}\dim B_2(A_n/\langle x^d+y^d\rangle )[l] =\dim  \Omega^{1}_{x^d+y^d=0}/{\bold d}\Omega^{0}_{x^d+y^d=0}[l], \end{equation} so putting (2)-(5) together we can conclude that $$\dim B_2(A_n/\langle P\rangle)[l]=\dim  \Omega^{1}_{P=0}/{\bold d}\Omega^{0}_{P=0}[l],$$ so the map $$B_{2}(A_{n}/ \langle P \rangle) \twoheadrightarrow \gr B_{2}(A_{n}/ \langle P-1 \rangle) $$ is an isomorphism, and so is the composite map $$\gr \phi: B_{2}(A_{n}/ \langle P \rangle)\to \Omega^{1}_{P=0}/{\bold d}\Omega^{0}_{P=0}..$$
\end{proof}

\begin{conjecture}
For $n=4$, the map $\gr \phi$ from the statement of Theorem \ref{main} is surjective with a finite dimensional kernel; i.e. $ B_{2}(A_{n}/ \langle P \rangle)[l]\cong \Omega^{1}_{P=0}/{\bold d}\Omega^{0}_{P=0}[l]$ for $l$ large enough.
\end{conjecture}

The proof would be analogous to the $n=2,3$ case, with some changes. First one needs to prove the isomorphism  $$B_{2}(A_{n}/ \langle P \rangle)[l]\cong \Omega^{1}_{P=0}/{\bold d}\Omega^{0}_{P=0}[l]$$ for $l>>0$ and a specific $P$. In this case, the polynomial $P=x^d+y^d$ is not generic enough; MAGMA computations of the dimensions of graded components of $B_{2}(A_{n}/ \langle P \rangle)$ show that they are larger then those of $\Omega^{1}_{P=0}/{\bold d}\Omega^{0}_{P=0}$ at low degrees. The same is true for $P=x^d+y^d+z^d$. The polynomial $P=x^d+y^d+z^d+w^d$ is, according to MAGMA computations, generic enough as the beginning of the Poincar\' e-Hilbert series for $B_{2}(A_{n}/ \langle P \rangle)$ and  $\Omega^{1}_{P=0}/{\bold d}\Omega^{0}_{P=0}$ match. The Hilbert-Poincar\' e series for $\Omega^{1}_{P=0}/{\bold d}\Omega^{0}_{P=0}$ is easily calculated, as in Corollary \ref{isomo2} and Corollary \ref{isomo3}, to be $$h_{\Omega^{1}_{P=0}/{\bold d}\Omega^{0}_{P=0}}(t)=\frac{((1-t)^4-1+4t-4t^{d+1}+t^{2d})}{(1-t)^{4}}.$$ We conjecture that the series for $B_{2}(A_{4}/ \langle P \rangle)$ is $$h_{B_{2}(A_{4}/\langle P \rangle)}(t)=\frac{((1-t)^4-1+4t-4t^{d+1}+t^{2d})}{(1-t)^{4}} + \left(\frac{t(1-t^{d-1})}{(1-t)}\right)^{4}.$$ The second term of it is a polynomial (so it influences only finitely many degrees), and it corresponds to the facts that for $n=4$, the analogue of a finite dimensional space  $\Omega^{\mathrm{3}}/{\bold d}\Omega^{\mathrm{2}}$ also appears here, and to the fact that the smooth variety corresponding to the abelianization of $A_{n}/\left<P-1 \right>$ will have cohomology of dimension $(d-1)^4$ at degree $3$.  

After this, generalization to a generic $P$ should be proved as it was in the $n=2,3$ case in Theorem \ref{main}. 

For $n>4$, a similar statement should hold, but different techniques should be considered for a precise statement and a proof.

\appendix


\section{Computational Results}
Here is a series of tables obtained in the first phase of the project using MAGMA, containing dimensions of the graded pieces of $B_{i}(A)$.The columns are labeled by $l$, the degree or graded component we are interested in, and rows by the $B_i$.  Each row, in turn, gives coefficients of the Hilbert-Poincar\' e series of the appropriate $B_i$. We used similar tables constructed for $B_{2}(A_{n}/\left< x^d+y^d\right>), n=2,3$, to derive conjectures about bases of these spaces. These conjectures became Theorems \ref{basis2} and \ref{basisb2a3}.

\begin{table}[htbp]
\begin{center}
\caption{$B_{i}(A_{2}/\left< x^2+y^2 \right>)$}
\begin{tabular}{|c||c|c|c|c|c|c|}\hline\hline
$B_{i}[l]$ & \textbf{1} & \textbf{2} & \textbf{3} & \textbf{4} & \textbf{5} & \textbf{6} \\ \hline\hline
\textbf{$B_{3}$} & 0 & 0 & 2 & 0 & 0 & 0 \\ \hline
\textbf{$B_{4}$} & 0 & 0 & 0 & 2 & 0 & 0 \\ \hline
\textbf{$B_{5}$} & 0 & 0 & 0 & 0 & 4 & 0 \\ \hline
\end{tabular}
\end{center}
\label{A2Degree2}
\end{table}




\begin{table}[h]
\begin{center}
\caption{$B_{i}(A_{3}/\left< x^2+y^2 \right>)$}
\begin{tabular}{|c||c|c|c|c|c|c|}\hline\hline
$B_{i}[l]$ & \textbf{1} & \textbf{2} & \textbf{3} & \textbf{4} & \textbf{5} & \textbf{6} \\ \hline\hline
\textbf{$B_{3}$} & 0 & 0 & 8 & 15 & 19 & 23 \\ \hline
\textbf{$B_{4}$} & 0 & 0 & 0 & 17 & 45 & 59 \\ \hline
\end{tabular}
\end{center}
\label{A3Degree2term2}
\end{table}

\begin{table}[h]
\begin{center}
\caption{$B_{i}(A_{3}/\left< x^2+y^2+z^2 \right>)$}
\begin{tabular}{|c||c|c|c|c|c|c|}\hline\hline
$B_{i}[l]$ & \textbf{1} & \textbf{2} & \textbf{3} & \textbf{4} & \textbf{5} & \textbf{6} \\ \hline\hline
\textbf{$B_{3}$} & 0 & 0 & 8 & 15 & 16 & 20 \\ \hline
\textbf{$B_{4}$} & 0 & 0 & 0 & 18 & 45 & 48 \\ \hline
\end{tabular}
\end{center}
\label{A3Degree2term3}
\end{table}

\begin{table}[h]
\begin{center}
\caption{$B_{i}(A_{3}/\left< x^3+y^3 \right>)$}
\begin{tabular}{|c||c|c|c|c|c|c|}\hline\hline
$B_{i}[l]$ & \textbf{1} & \textbf{2} & \textbf{3} & \textbf{4} & \textbf{5} & \textbf{6} \\ \hline\hline
\textbf{$B_{3}$} & 0 & 0 & 8 & 24 & 39 & 48 \\ \hline
\textbf{$B_{4}$} & 0 & 0 & 0 & 18 & 71  & 133 \\ \hline
\end{tabular}
\end{center}
\label{A3Degree3term2}
\end{table}

\begin{table}[h]
\begin{center}
\caption{$B_{i}(A_{3}/\left< x^3+y^3+z^3 \right>)$}
\begin{tabular}{|c||c|c|c|c|c|c|}\hline\hline
$B_{i}[l]$ & \textbf{1} & \textbf{2} & \textbf{3} & \textbf{4} & \textbf{5} & \textbf{6} \\ \hline\hline
\textbf{$B_{3}$} & 0 & 0 & 8 & 24 & 39 & 45 \\ \hline
\textbf{$B_{4}$} & 0 & 0 & 0 & 18 & 72  & 135 \\ \hline
\end{tabular}
\end{center}
\label{A3Degree3term3}
\end{table}






\begin{table}[!ht]
\begin{center}
\caption{$B_{i}(A_{4}/\left< x^2+y^2+z^2+w^2 \right>)$}
\begin{tabular}{|c||c|c|c|c|c|c|c|}\hline\hline
$B_{i}[l]$ & \textbf{1} & \textbf{2} & \textbf{3} & \textbf{4} & \textbf{5} & \textbf{6} \\ \hline\hline
\textbf{$B_{2}$} & 0 & 6 & 16 & 31 & 48 & 70 \\ \hline
\textbf{$B_{3}$} & 0 & 0 & 20 & 64 & 124 &  \\ \hline
\textbf{$B_{4}$} & 0 & 0 & 0 & 60 &   & \\ \hline
\end{tabular}
\end{center}
\label{A4Degree2}
\end{table}

\begin{table}[!ht]
\begin{center}
\caption{$B_{i}(A_{4}/\left< x^3+y^3+z^3+w^3 \right>)$}
\begin{tabular}{|c||c|c|c|c|c|c|}\hline\hline
$B_{i}[l]$ & \textbf{1} & \textbf{2} & \textbf{3} & \textbf{4} & \textbf{5} \\ \hline\hline
\textbf{$B_{2}$} & 0 & 6 & 20 & 42 & 72 \\ \hline
\textbf{$B_{3}$} & 0 & 0 & 20 & 80 & 188 \\ \hline
\end{tabular}
\end{center}
\label{A4Degree3}
\end{table}

\end{document}